\documentclass[11pt, reqno]{amsart}
\usepackage{amsmath, amssymb, bm, url, enumitem}
\usepackage{color}
\usepackage{graphicx}
\usepackage{tikz}
\usepackage[all]{xy}
\allowdisplaybreaks

\newtheorem{theorem}{Theorem}[section]
\newtheorem{lemma}[theorem]{Lemma}
\newtheorem{proposition}[theorem]{Proposition}
\newtheorem{corollary}[theorem]{Corollary}

\theoremstyle{definition}
\newtheorem{definition}[theorem]{Definition}
\newtheorem{example}[theorem]{Example}

\theoremstyle{remark}
\newtheorem{remark}[theorem]{Remark}
\numberwithin{equation}{section}

\newcommand{\Z}{\mathbb{Z}}

\newcommand{\C}{\mathfrak{C}}

\renewcommand{\L}{\mathcal{L}}

\renewcommand{\a}{\mathsf{a}}%\a{'}{b}

\newcommand{\vlk}{\mathrm{vlk}}

\newcommand{\sign}{\operatorname{sign}}
\newcommand{\Hom}{\mathrm{Hom}}

\newcommand{\ang}[1]{\langle #1 \rangle}

\newcommand{\asc}{\mathrm{asc}}
\newcommand{\des}{\mathrm{des}}
\newcommand{\Arf}{\operatorname{Arf}}

\title[]{On the determinant of checkerboard colorable virtual knots}

\author{Tomoaki Hatano}
\address{}
\email{}

\author{Yuta Nozaki}
\address{
Faculty of Environment and Information Sciences, Yokohama National University \\
79-7 Tokiwadai, Hodogaya-ku, Yokohama, 240-8501 \\
Japan\vspace{-0.6em}}
\address{
WPI-SKCM$^2$, Hiroshima University \\
1-3-1 Kagamiyama, Higashi-Hiroshima, Hiroshima 739-8526 \\
Japan}
\email{nozaki-yuta-vn@ynu.ac.jp}

\makeatletter
\@namedef{subjclassname@2020}{\textup{2020} Mathematics Subject Classification}
\makeatother


\begin{document}
\maketitle

\begin{abstract}
For classical knots, Murasugi showed that the determinant modulo $8$ is classified by the Arf invariant.
Boden and Karimi introduced a determinant for checkerboard colorable virtual knots.
We prove that this determinant modulo $8$ is classified by the coefficient of $z^2$ in the ascending polynomial, an extension of the Conway polynomial for classical knots.
\end{abstract}

%\setcounter{tocdepth}{1}
%\tableofcontents

%%%%%%%%%
\section{Introduction}
\label{sec:Introduction}

For a knot in $S^3$, let $\Delta_K(t)$ denote the Conway-normalized Alexander polynomial, that is, $\Delta_K(t)=\det(t^{1/2}V-t^{-1/2}V^T)$ for a Seifert matrix $V$ of $K$.
The special value $\Delta_K(-1)$ is known to be an odd integer, and its absolute value $\det K$ is called the \emph{determinant} of $K$.
Here $\Delta_K(t)$ is related to the Conway polynomial $\nabla_K(z)$ of $K$ via $\Delta_K(t)=\nabla_K(t^{-1/2}-t^{1/2})$.
Let $\Arf K\in \Z/2\Z$ denote the \emph{Arf invariant} of $K$.
Murasugi~\cite{Mur69} established the following relation, which is now well known:
\begin{align}
\det K \equiv
\begin{cases}
 \pm 1 \mod 8 & \text{if $\Arf K=0$,}\\
 \pm 3 \mod 8 & \text{if $\Arf K=1$.}
\end{cases}
\label{eq:classical}
\end{align}

\begin{figure}[h]
 \centering \includegraphics[width=0.8\textwidth]{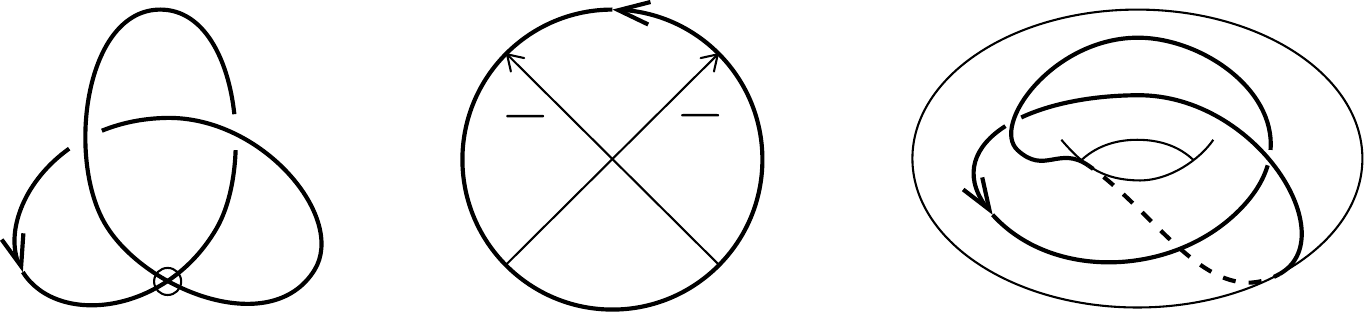}
 \caption{A virtual knot diagram, the corresponding Gauss diagram, and (projection of) a knot in a thickened torus.}
 \label{fig:virtual_trefoil}
\end{figure}

In this paper, we extend the above classical result to a certain class of virtual knots.
A virtual knot is a generalization of classical knots, defined by a virtual knot diagram containing both real and virtual crossings, or by a Gauss diagram as illustrated in Figure~\ref{fig:virtual_trefoil}.
It is also defined to be a certain equivalence class of a knot in $\Sigma\times[0,1]$, where $\Sigma$ is an orientable closed surface.
For its precise definition and related topics, we refer the reader to \cite{Kau99}, \cite{BGHNW17}, \cite{BoKa22}, and \cite{CDM12}. 

Boden, Chrisman, and Karimi~\cite{BCK22} introduced a determinant of a pair $(L,F)$, where $L$ is a link in $\Sigma\times[0,1]$ with $[L]=0$ in $H_1(\Sigma\times[0,1];\Z/2\Z)$ and $F$ is a spanning surface for $L$.
Their determinant can be applied to \emph{checkerboard colorable virtual links} introduced by Kamada~\cite{Kam02} since they admit spanning surfaces.
Boden and Karimi~\cite{BoKa22} defined another determinant for these links, which is independent of the choice of a spanning surface.
Now, it is natural to expect that the Arf invariant for classical knots can be extended to checkerboard colorable virtual knots so that the extension and the determinant in \cite{BoKa22} fit into the equality \eqref{eq:classical}.

When a knot $K$ satisfies $[K]=0$ in $\Z$-coefficients and $F$ is a Seifert surface of $K$, the Arf invariant is defined for $(K,F)$ and used as an obstruction to virtual concordance of Seifert surfaces by Chrisman, Mukherjee~\cite{ChMu23} and Boden, Karimi~\cite{BoKa24}.
The Arf invariant for $(K,F)$, however, does not satisfy \eqref{eq:classical} above (see Example~\ref{ex:Green}).

We here focus on the fact that $\Arf K = v_2(K)\bmod 2$ for classical knots $K$, where $v_2(K)$ denotes the coefficient of $z^2$ in the Conway polynomial $\nabla_K(z)$.
For a Gauss diagram $G$ of a virtual knot with a basepoint $b$ on the circle disjoint from chords, let $\nabla_\asc(G_b)(z) \in \Z[z]$ (resp.\ $\nabla_\des(G_b)(z)$) denote the ascending (resp.\ descending) polynomial introduced by Chmutov, Khoury, and Rossi~\cite{CKR09}.
For classical knots, $\nabla_\asc(G_b)(z)$ and $\nabla_\des(G_b)(z)$ do not depend on the choice of a Gauss diagram and basepoint, and they coincide with the Conway polynomial.

We write $v_{2,1}(G_b)$ (resp.\ $v_{2,2}(G_b)$) for the coefficient of $z^2$ in $\nabla_\asc(G_b)(z)$ (resp.\ $\nabla_\des(G_b)(z)$) and show that they behave well for a certain class of virtual knots containing checkerboard colorable virtual knots.

\begin{theorem}
\label{thm:main}
Let $p$ be a non-negative integer and let $G$ be a Gauss diagram of a mod $p$ almost classical knot.
Then, for $j=1,2$, $v_{2,j}(G_b) \bmod p$ do not depend on the choice of a basepoint $b$.
Moreover, $v_{2,1}(G_b)\equiv v_{2,2}(G_b) \bmod p$.
\end{theorem}

This theorem allows us to define $v_2(K)\in \Z/2\Z$ for checkerboard colorable virtual knots $K$ by $v_2(K)=v_{2,1}(G_b)\equiv v_{2,2}(G_b) \bmod 2$.
The purpose of this paper is to prove the following relation between $\det K$ in \cite{BoKa22} and our invariant $v_2(K)$.

\begin{corollary}
\label{cor:det}
Let $K$ be a checkerboard colorable virtual knot.
Then,
\[
\det K \equiv
\begin{cases}
 \pm 1 \mod 8 & \text{if $v_2(K)=0$,}\\
 \pm 3 \mod 8 & \text{if $v_2(K)=1$.}
\end{cases}
\]
\end{corollary}

Note here that a similar result for almost classical knots is mentioned in \cite[Section~4]{BoKa24}, but the definition of the determinant is different from \cite{BoKa22} and ours.
This difference arises from the difference in Gordon-Litherland linking forms (see a discussion after \cite[Definition~3.1]{BoKa23}).

By Corollary~\ref{cor:det}, while the Arf invariant derived from Seifert matrices coincides with the coefficient of $z^2$ in the Conway polynomial modulo $2$ in classical knot theory, they play different roles in virtual knot theory.
The former is an obstruction to concordance, and the latter controls the determinant modulo $8$.
Finally, this paper is based on the first author's master thesis at Hiroshima University, written in Japanese.

\subsection*{Acknowledgments}
The authors are grateful to Kodai Wada for his careful reading of the manuscript, and to Noboru Ito and Yuka Kotorii for their helpful comments.
They also wish to express their thanks to the anonymous referee for valuable comments.
This study was supported in part by JSPS KAKENHI Grant Numbers JP20K14317 and JP23K12974.

%%%%%%%
\section{Preliminaries}
\label{section:Preliminaries}

\subsection{mod $p$ almost classical knots}
We review basic definitions concerning virtual links according to \cite{BGHNW17}.
A \emph{virtual link} is an equivalence class of virtual link diagrams, where two diagrams $D$ and $D'$ are equivalent if there exists a finite sequence of Reidemeister moves and virtual moves from $D$ to $D'$.
Also, $D$ and $D'$ are said to be \emph{welded equivalent} if they are related by the Reidemeister moves, virtual moves, and forbidden overpass move in \cite[Figure~1]{BGHNW17}.
A \emph{short arc} (resp.\ \emph{long arc}) of a virtual link diagram is an arc connecting two consecutive crossings (resp.\ two under-crossings).
For instance, the diagram on the left of Figure~\ref{fig:virtual_trefoil} has six short arcs and two long arcs.

We recall the following definition from \cite[Section~5]{BGHNW17}.

\begin{definition}
\label{def:almost_classical}
Let $p$ be a non-negative integer.
A diagram $D$ of oriented virtual link is said to be \emph{mod $p$ Alexander numberable} if there exists an assignment of integers to the short arcs of $D$ such that the four integers $\lambda_a, \lambda_b, \lambda_c, \lambda_d$ around a crossing as illustrated in Figure~\ref{fig:Alexander_numbering} satisfy
\begin{itemize}
 \item $\lambda_b\equiv \lambda_c$ and $\lambda_a\equiv \lambda_d\equiv \lambda_b+1 \bmod p$ if the crossing is real,
 \item $\lambda_a= \lambda_c$ and $\lambda_b= \lambda_d$ if the crossing is virtual.
\end{itemize}
Moreover, an oriented virtual link $L$ is said to be \emph{mod $p$ almost classical} if $L$ admits a mod $p$ Alexander numberable diagram.
As special cases, a mod $2$ almost classical link is said to be \emph{checkerboard colorable} in \cite{Kam02}, and mod $0$ almost classical link is said to be \emph{almost classical} in \cite{SiWi06}.

\end{definition}

\begin{figure}[h]
 \centering \includegraphics[width=0.8\textwidth]{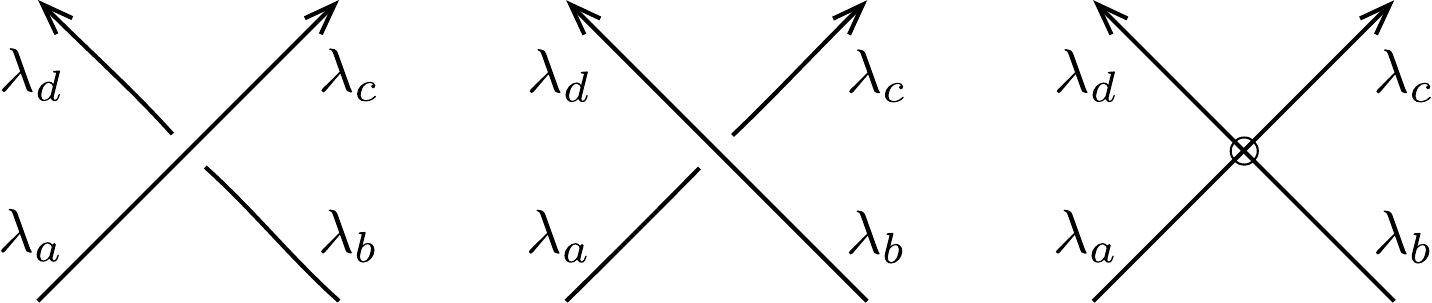}
 \caption{Colorings of short arcs around positive/negative crossings and virtual crossing.}
 \label{fig:Alexander_numbering}
\end{figure}

Being mod $p$ almost classical is characterized by the first homology group of a thickened surface as stated in the next theorem, which is shown in much the same way as \cite[Theorem~6.1]{BGHNW17}.
See also an argument written before Theorem~5.2 in \cite{BGHNW17}.

\begin{theorem}
\label{thm:Carter_surface}
For an oriented virtual link $L$, the following are equivalent.
\begin{enumerate}[label=\textup{(\alph*)}]
\item $L$ is mod $p$ almost classical.
\item $[L] \in H(\Sigma\times[0,1];\Z/p\Z)$ is trivial, where $\Sigma$ is the Carter surface of a diagram of $L$.
\end{enumerate}
\end{theorem}

In particular, a checkerboard colorable virtual link $L$ actually admits a checkerboard coloring in the sense of \cite[Section~2]{Kam02} (or \cite[Introduction]{ChPa07}), and it has a spanning surface in $\Sigma\times [0,1]$.

\begin{example}
The virtual knot in Figure~\ref{fig:virtual_trefoil} is not checkerboard colorable.
Indeed, one can check that the diagram in Figure~\ref{fig:virtual_trefoil} is not mod $2$ Alexander numberable and the crossing number of this diagram is minimal, and thus \cite[Theorem~8.3]{BGHNW17} implies that it is not checkerboard colorable.

In Green's table~\cite{Green}, $K=4.90$ is checkerboard colorable as illustrated in Figure~\ref{fig:4_90}.
On the other hand, one can see that it is not almost classical.
\end{example}

Recall here that every oriented virtual link is represented by a Gauss diagram, which is a chord diagram with oriented chords (called arrows) and with signs $+1$ or $-1$ assigned to each chord.
See, for example, \cite{CKR09}.
Throughout this paper, orientations of circles of Gauss diagrams (and arrow diagrams in Section~\ref{sec:ascending_descending}) are assumed to be counterclockwise.

We recall from \cite[Section~1]{BGHNW17} the definition of the index of a chord.
See also \cite{SaTa14}.

\begin{definition}
\label{def:index}
Let $G$ be a Gauss diagram of a virtual knot.
The \emph{index} $I(c)\in \Z$ of a chord $c$ of $G$ with sign $\varepsilon$ is defined by $I(c)=\varepsilon(r_{+} -r_{-} +l_{-} -l_{+})$, where $r_\pm$ (resp.\ $l_\pm$) denote the numbers of chords with sign $\pm$ intersecting $c$ from left to right (resp.\ from right to left) when the head of $c$ is oriented up.
\end{definition}

The following lemma is stated after Definition~5.1 in \cite{BGHNW17}.

\begin{lemma}
\label{lem:index}
Let $G$ be a Gauss diagram of an oriented virtual knot diagram $D$.
Then $D$ is mod $p$ Alexander numberable if and only of $I(c)\equiv 0\bmod p$ for any chord $c$ of $G$.
\end{lemma}

\begin{figure}[h]
 \centering \includegraphics[width=0.8\textwidth]{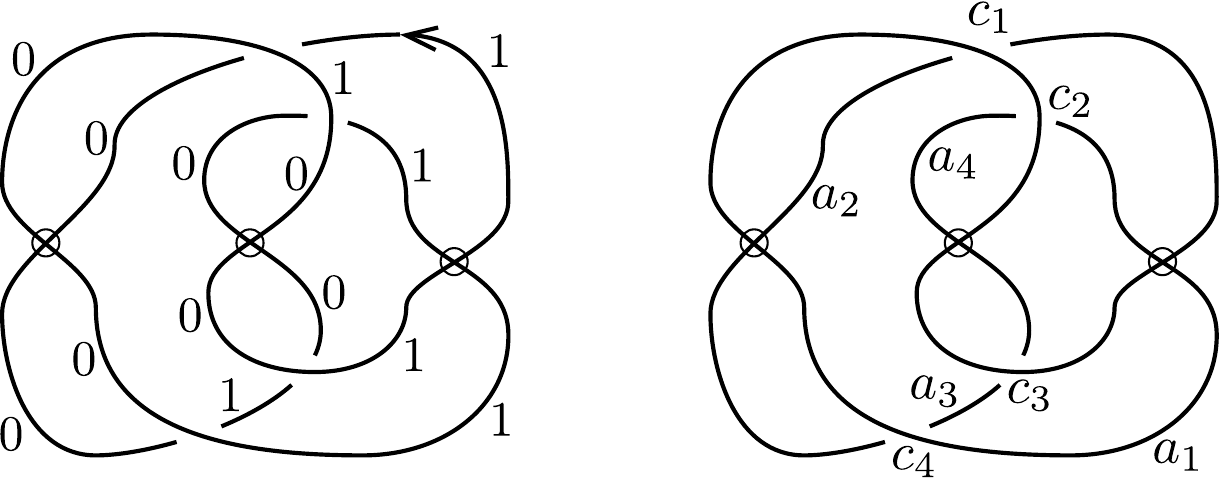}
 \caption{A mod $2$ Alexander numbering of a diagram of $4.90$ and labels of crossings and long arcs.}
 \label{fig:4_90}
\end{figure}

%%%
\subsection{Determinant for checkerboard colorable links}
We recall from \cite[Section~3]{BoKa22} the definition of the coloring matrix.

\begin{definition}
Let $D$ be a checkerboard colorable diagram with real crossings $c_1,\dots,c_n$ and long arcs $a_1,\dots,a_m$.
The \emph{coloring matrix} $B(D)=(b_{ij})$ is the $n\times m$ matrix such that $b_{ij}=1$ if $a_j$ is both the over-crossing arc and one of the under-crossing arcs at $c_i$, otherwise
\[
b_{ij}=
\begin{cases}
2 & \text{if $a_j$ is the over-crossing arc at $c_i$,}\\
-1 & \text{if $a_j$ is one of the under-crossing arcs at $c_i$,}\\
0 & \text{otherwise.}
\end{cases}
\]
\end{definition}

Now, we can define a determinant, which is an essential ingredient of this paper.

\begin{definition}
\label{def:determinant}
The \emph{determinant} $\det L$ of checkerboard colorable virtual link $L$ is defined to be the absolute value of the determinant of an $(n-1)\times(n-1)$ minor of $B(D)$, which is well-defined due to \cite[Proposition~3.1]{BoKa22}.
\end{definition}

The determinant is actually an invariant of welded links since the coloring matrix is invariant under the forbidden overpass move in \cite[Figure~1]{BGHNW17}.
We use this fact in the proof of Lemma~\ref{lem:det1}.

\begin{example}
For the labels of the diagram of $K=4.90$ on the right of Figure~\ref{fig:4_90}, we have
\[
B(D)=
\begin{pmatrix}
1 & -1 & 0 & 0 \\
1 & 0 & 0 & -1 \\
2 & 0 & -1 & -1 \\
2 & -1 & -1 & 0 \\
\end{pmatrix}.
\]
Therefore, $\det K =1$. 
\end{example}

To show a kind of skein relation for $\det L$ in the next subsection, we give another interpretation of $\det L$ in terms of a mock Seifert matrix introduced in \cite[Definition~3.5]{BoKa23}.
Let $\Sigma$ be an orientable closed surface.
For two disjoint oriented knots $K_1$ and $K_2$ in $\Sigma\times[0,1]$, we define the \emph{virtual linking number} $\vlk(K_1,K_2) \in \Z$ by the identity
\[
[K_2]=\vlk(K_1,K_2)[\mu] \in H_1(\Sigma\times[0,1]\setminus K_1, \Sigma\times\{1\};\Z)\cong \Z.
\]
We recall the Gordon-Litherland form $\L_F$ from \cite{BoKa23}, where $F$ is a compact connected (unoriented) surface in the interior of $\Sigma\times[0,1]$.
Let $\widetilde{F}$ be the boundary of a neighborhood of $F$ and $\pi\colon \widetilde{F}\to F$ be the induced double cover.
One can define a transfer map $\tau\colon H_1(F;\Z)\to H_1(\widetilde{F};\Z)$, and then the Gordon-Litherland form $\L_F\colon H_1(F)\times H_1(F)\to \Z$ is defined by $\L_F(\alpha,\beta) = \vlk(\tau(\alpha),\beta)$.

\begin{definition}
Let $L$ be a link in $\Sigma\times[0,1]$ with $[L]=0\in H_1(\Sigma\times[0,1];\Z/2\Z)$ and let $F$ be a spanning surface for $L$, that is, $F$ is (not necessarily orientable) connected compact surface whose boundary is $L$.
A matrix for the Gordon-Litherland form $\L_F$ associated with $F$ is called a \emph{mock Seifert matrix} for $L$.
\end{definition}

Let $L$ be a checkerboard colorable virtual link, $D$ a diagram of $L$, and $\Sigma$ the Carter surface of $D$.
We write for $X$ the quotient space obtained from $\Sigma\times[0,1]$ by identifying $\Sigma\times\{1\}$ with a point.
Let $X_2$ be the double cover of $X$ branched over $L$ and $X_\infty$ the infinite cyclic cover of $X\setminus L$.
It is shown in \cite[Section~3]{BoKa22} that we can compute the first elementary ideal $\mathcal{E}_1$ of the Alexander module $H_1(X_\infty;\Z)$ over the ring $\Z[t,t^{-1}]$ from a matrix $A(D)$ obtained by Fox's free derivative.
As mentioned before Proposition~3.1 in \cite{BoKa22}, the coloring matrix $B(D)$ is obtained from $A(D)$ by substituting $t=-1$ (up to sign for any given row).
Therefore, one can compute the first elementary ideal of $H_1(X_2;\Z)$ over the ring $\Z[-1]=\Z$ from $B(D)$.
On the other hand, it is shown in \cite[Theorem~3.9]{BoKa23} that a mock Seifert matrix for $L$ is a presentation matrix of $H_1(X_2;\Z)$.
Therefore, we obtain the following consequence (see \cite[Remark~3.11]{BoKa23}).

\begin{proposition}
Let $S$ be a mock Seifert matrix for $L$.
Then, $\det L = |{\det S}|$ holds.
\end{proposition}

\begin{remark}
For an almost classical oriented link $L$, the Alexander polynomial $\Delta_L(t)$ is defined in \cite[Section~7]{BGHNW17} (see also \cite{NNST12}).
It is shown in \cite[Proposition~4.1]{BoKa22} that $\det L =|\Delta_L(-1)|$ holds.
\end{remark}

\begin{figure}[h]
 \centering \includegraphics[width=0.6\textwidth]{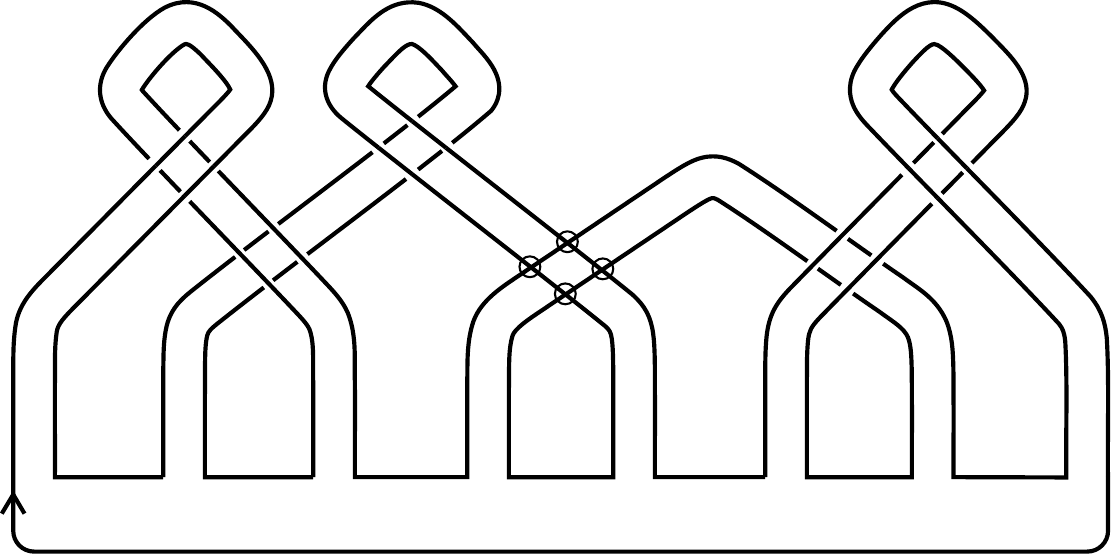}
 \caption{The virtual knot $6.87548$ in Green's table~\cite{Green} (see \cite[Figure~20]{BGHNW17} for its Gauss diagram).}
 \label{fig:687548}
\end{figure}

\begin{example}
\label{ex:mock}
Let $K$ be the virtual knot illustrated in Figure~\ref{fig:687548}.
Using a Seifert surface in \cite[Figure~14]{Chr19} as a spanning surface, we have
\[
S = 
\begin{pmatrix}
-2 & -1 & 0 & 0 \\
-1 & 2 & 1 & 0 \\
0 & -1 & 0 & 1 \\
0 & 0 & 1 & 2
\end{pmatrix},
\]
which is also obtained by adding $V^+$ and $V^-$ in \cite[Example~(6.87548) in Section~6]{Chr19}.
Then one has $\det K =1$.
Note that this result is compatible with \cite[Table~2]{BGHNW17}.
\end{example}

%%%
\subsection{Skein relation}
We prove a kind of skein relation for $\det L$ similar to \cite[Theorem~7.11]{BGHNW17} for the Alexander polynomial of almost classical links.
The idea of the proof is based on \cite[Lemma~3]{Gil82}.
Recall from Theorem~\ref{thm:Carter_surface} that a checkerboard colorable virtual link has a spanning surface in $\Sigma\times [0,1]$.

\begin{theorem}
\label{thm:det_skein}
Let $L_+$, $L_-$ and $L_0$ be oriented virtual links in Figure~\ref{fig:skein_triple}.
Suppose that one of them is checkerboard colorable \textup{(}then the others are also checkerboard colorable\textup{)} and consider spanning surfaces in Figure~\ref{fig:spanning_surface}.
Then, mock Seifert matrices $S_+$, $S_-$, $S_0$ associated with the spanning surfaces satisfy
\[
\det S_+ -\det S_- = 2\det S_0.
\]
\end{theorem}

\begin{figure}[h]
 \centering \includegraphics[width=0.7\textwidth]{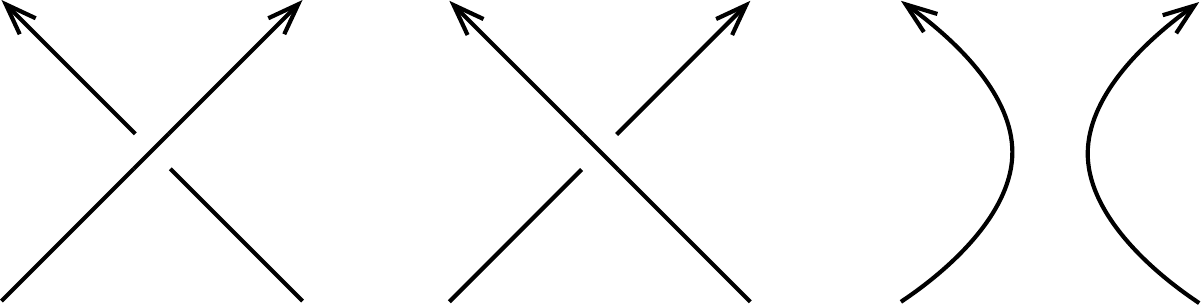}
 \caption{Oriented virtual links $L_+$, $L_-$, $L_0$ that are identical except within a $3$-ball.}
 \label{fig:skein_triple}
\end{figure}

\begin{figure}[h]
 \centering \includegraphics[width=0.7\textwidth]{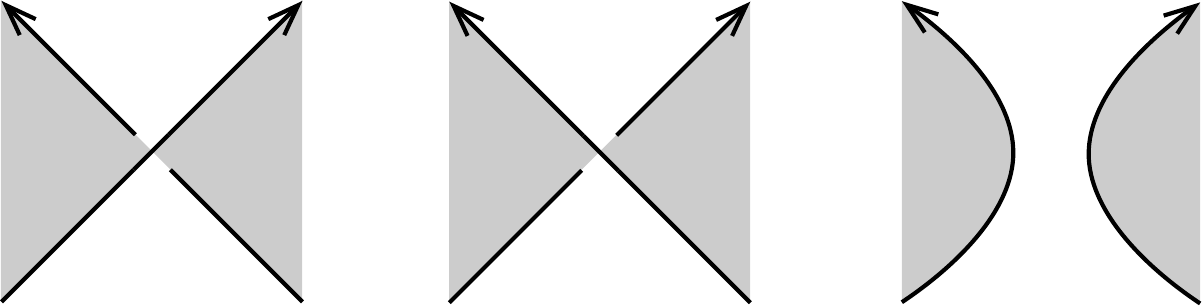}
 \caption{Spanning surfaces $F_+$, $F_-$, $F_0$ that are identical except within a $3$-ball.}
 \label{fig:spanning_surface}
\end{figure}

\begin{remark}
When $F_0$ is disconnected, we obtain a spanning surface by attaching a tube to $F_0$.
The mock Seifert matrix for the resulting spanning surface satisfies Theorem~\ref{thm:det_skein} because then $\det S_+ =\det S_-$ and $\det S_0=0$.
\end{remark}

\begin{proof}[Proof of Theorem~\ref{thm:det_skein}]
Fix a basis of $H_1(F_0;\Z)$ and let $S_0$ be the corresponding mock Seifert surface for $F_0$.
Since $F_+$ is obtained from the connected surface $F_0$ by attaching a band, there exists a loop $\ell$ on $F_+$ passing through the band.
By the Mayer-Vietoris exact sequence, the union of the basis of $H_1(F_0;\Z)$ and $[\ell] \in H_1(F_+;\Z)$ is a basis of $H_1(F_+;\Z)$.
Now, the union also gives a basis of $H_1(F_-;\Z)$.
Let $S_\pm$ be mock Seifert surfaces for $F_\pm$ with respect to the bases.
Then one can see that
\[
S_+=
\begin{pmatrix}
S_0 & \bm{x} \\
\bm{y}^T & a
\end{pmatrix},\quad
S_-=
\begin{pmatrix}
S_0 & \bm{x} \\
\bm{y}^T & a-2
\end{pmatrix}
\]
for some $a\in \Z$ and integral column vectors $\bm{x}, \bm{y}$.
We have
\begin{align*}
\det S_- &= \det
\begin{pmatrix}
S_0 & \bm{x} \\
\bm{y}^T & a-2
\end{pmatrix} \\
&= \det
\begin{pmatrix}
S_0 & \bm{x} \\
\bm{y}^T & a
\end{pmatrix}
+\det
\begin{pmatrix}
S_0 & \bm{0} \\
\bm{y}^T & -2
\end{pmatrix} \\
&= \det S_+ -2\det S_0.
\end{align*}
This completes the proof.
\end{proof}

%%%%%
\section{Ascending and descending polynomials}
\label{sec:ascending_descending}

We review basic definitions concerning the ascending and descending polynomials according to \cite{CKR09}.
An \emph{arrow diagram} is a based chord diagram with oriented chords (called arrows).
A \emph{homomorphism} $\varphi$ from an arrow diagram $A$ to a based Gauss diagram $G$ is an orientation-preserving homeomorphism of the circle of $A$ to the circle of $G$ which maps the basepoint to the basepoint and induces an injective map of arrows of $A$ to arrows of $G$ respecting the orientation of the arrows.

\begin{definition}
For an arrow diagram $A$ and a based Gauss diagram $G$, define the \emph{pairing} $\ang{A,G}\in \Z$ by
\[
\ang{A,G}=\sum_{\varphi\in\Hom(A,G)} \prod_{\text{$\a$:\,arrow in $A$}}\sign(\varphi(\a)).
\]
\end{definition}

We recall some terminology from \cite[Definitions~4.1 and 4.3]{CKR09}.
A chord diagram (or arrow diagram) is said to be \emph{$k$-component} if the curve obtained from the diagram by taking the parallel double of each chord consists of $k$ circles.
Let $A$ be a $1$-component arrow diagram and consider traversing the curve obtained by parallel doubling, starting from the basepoint and following the orientation of the diagram.
The diagram $A$ is said to be \emph{ascending} (resp.\ \emph{descending}) if, for each arrow $\a$, the direction in which $\a$ is first traversed is opposite to (resp.\ the same as) the direction of $\a$.

\begin{definition}
For a non-negative integer $n$, the \emph{Conway combination} $\C_{2n}$ (resp.\ $\C'_{2n}$) is defined to be the sum of $1$-component ascending (resp.\ descending) arrow diagrams with $2n$ arrows on a circle.
Similarly, $\C_{2n+1}$ (resp.\ $\C'_{2n+1}$) is defined to be the sum of $1$-component ascending (resp.\ descending) diagrams with $2n+1$ arrows on two circles.
\end{definition}

For instance, each of $\C_{1}$, $\C'_{1}$, $\C_{2}$, and $\C'_{2}$ consists of a single diagram illustrated in Figure~\ref{fig:Conway_combination}.

\begin{figure}[h]
 \centering \includegraphics[width=0.9\textwidth]{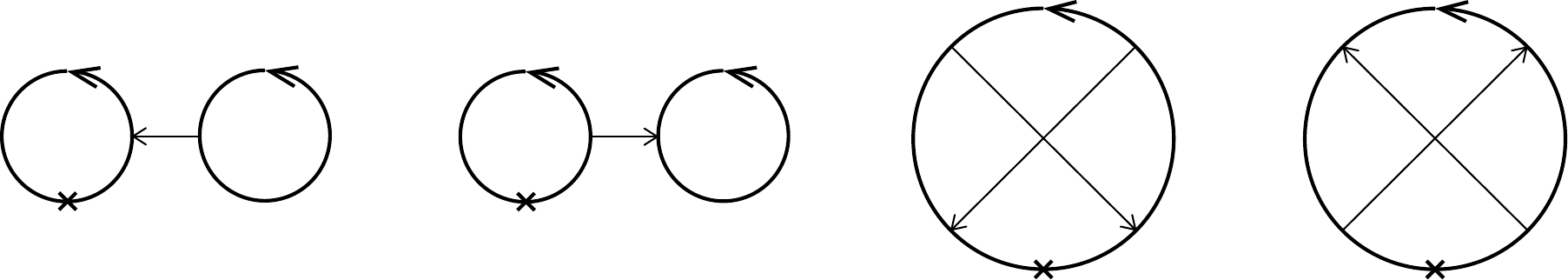}
 \caption{Conway combinations $\C_{1}$, $\C'_{1}$, $\C_{2}$, and $\C'_{2}$.}
 \label{fig:Conway_combination}
\end{figure}

\begin{definition}
For a based Gauss diagram $G$, define the \emph{ascending polynomial} $\nabla_\asc(G)(z)$ and \emph{descending polynomial} $\nabla_\des(G)(z) \in \Z[z]$ respectively by
\[
\nabla_\asc(G)(z)=\sum_{i\geq 0}\ang{\C_{2i},G}z^{2i}
\text{ and }
\nabla_\des(G)(z)=\sum_{i\geq 0}\ang{\C'_{2i},G}z^{2i}.
\]
\end{definition}

\begin{remark}
The polynomials $\nabla_\asc(G)(z)$ and $\nabla_\des(G)(z)$ depend on the choice of a basepoint of $G$ in general.
For instance, one can observe it for the Gauss diagram in Figure~\ref{fig:virtual_trefoil}.
This is because the virtual knot in Figure~\ref{fig:virtual_trefoil} is not (mod $0$) almost classical, which implies that the assumption of Theorems~\ref{thm:main} and \ref{thm:basepoint} is essential in their proofs.
\end{remark}

\begin{example}
\label{ex:Green}
Let $K$ be the (almost classical) virtual knot $6.87548$ appearing in Example~\ref{ex:mock}.
For the Gauss diagram in \cite[Figure~14]{Chr19} (with an arbitrary basepoint), the coefficients of $z^2$ in $\nabla_\asc(G)(z)$ and in $\nabla_\des(G)(z)$ are $-2$.
On the other hand, concerning determinant and Arf invariant in \cite{BoKa24}, one can see $\det(K,F)=5$ and $\Arf(q_{K,F}) =1$.
Indeed, these are computed from a Seifert matrix
\[
V^+ =
\begin{pmatrix}
-1 & 0 & 0 & 0 \\
-1 & 1 & 1 & 0 \\
0 & -1 & 0 & 1 \\
0 & 0 & 0 & 1
\end{pmatrix}
%V^- = 
%\begin{pmatrix}
%-1 & -1 & 0 & 0 \\
%0 & 1 & 0 & 0 \\
%0 & 0 & 0 & 0 \\
%0 & 0 & 1 & 1
%\end{pmatrix},
\]
written in \cite[Example (6.87548) in Section~6]{Chr19}.
See also \cite[Remark~3.5]{BoKa24}.
\end{example}

In \cite[Lemma~5.1 and (5.3)]{CKR09}, skein relations among pairings are shown for classical links.
In their proofs, no property of classical knots is used, and hence we obtain the following lemmas as extensions to virtual links.

\begin{lemma}
\label{lem:skein_even}
Let $D_+$, $D_-$, $D_0$ be a skein triple drawn in Figure~\ref{fig:skein_triple} such that $D_\pm$ have one component and $D_0$ has two components.
Let $G_+$, $G_-$, $G_0$ be the corresponding Gauss diagrams.
Then, for $n\geq 1$,
\begin{align*}
\ang{\C_{2n},G_+} - \ang{\C_{2n},G_-} &= \ang{\C_{2n-1},G_0}, \\
\ang{\C'_{2n},G_+} - \ang{\C'_{2n},G_-} &= \ang{\C'_{2n-1},G_0}.
\end{align*}
\end{lemma}

\begin{lemma}
\label{lem:skein_odd}
Let $D_+$, $D_-$, $D_0$ be a skein triple drawn in Figure~\ref{fig:skein_triple} such that $D_\pm$ have two components and $D_0$ has one component.
Let $G_+$, $G_-$, $G_0$ be the corresponding Gauss diagrams.
Then, for $n\geq 0$,
\begin{align*}
\ang{\C_{2n+1},G_+} - \ang{\C_{2n+1},G_-} &= \ang{\C_{2n},G_0}, \\
\ang{\C'_{2n+1},G_+} - \ang{\C'_{2n+1},G_-} &= \ang{\C'_{2n},G_0}.
\end{align*}
\end{lemma}

We extend the notion of warping degree from \cite{Shi10} to based virtual knots. 
We will use this notion in the proof of Theorem~\ref{thm:det}.

\begin{definition}
Let $D$ be a diagram of an oriented virtual knot with a basepoint.
The \emph{warping degree} $d(D)$ of $D$ is defined to be the number of crossings at which we first go through the under-arcs when we travel from the basepoint along the orientation of $D$.
Moreover, the \emph{warping degree} $d(G)$ of a based Gauss diagram $G$ is defined to be that of a virtual knot diagram whose Gauss diagram is $G$.
\end{definition}

\begin{example}
Let $D$ be the oriented virtual knot diagram drawn in Figure~\ref{fig:687548} with the basepoint at the bottom.
We then have $d(D)=10$.
\end{example}

A diagram $D$ is said to be \emph{descending} if $d(D)=0$.
It is shown in \cite[Proposition~2.2]{Sat18} that a descending diagram is welded equivalent to the trivial one.
This fact will be used in the proof of Lemma~\ref{lem:det1}.

%%%%%
\section{Main results}

This section is devoted to proving Theorem~\ref{thm:main} and Corollary~\ref{cor:det}.
We first show some lemmas.

\begin{lemma}
\label{lem:warp_zero}
Let $G$ be a based Gauss diagram with $d(G)=0$.
Then, for any positive integer $n$,
\[
\ang{\C_{2n},G_b}=\ang{\C'_{2n},G_b}=0.
\]
\end{lemma}

\begin{proof}
First, $\ang{\C_{2n},G_b}=0$ is obvious by definition.
Next, since $d(G)=0$, every descending subdiagram of $G$ is of the form in Figure~\ref{fig:d0}.
This is not $1$-component, and thus $\ang{\C'_{2n},G_b}=0$.
\end{proof}

\begin{figure}[h]
 \centering \includegraphics[width=0.2\textwidth]{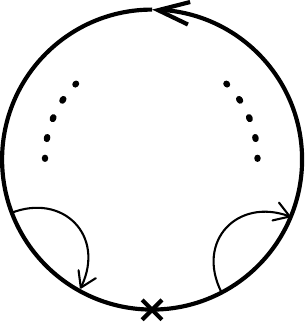}
 \caption{Descending diagram with isolated $2n$ arrows.}
 \label{fig:d0}
\end{figure}

\begin{lemma}
\label{lem:smoothing}
Let $G_b$ be a based Gauss diagram of mod $p$ almost classical oriented virtual knot and let $\a$ be an arrow.
Then $\a$ divides the circle into two arcs and suppose that the tail of any arrow intersecting with $\a$ is attached to the arc containing $b$.
Let $H_b$ be the $2$-component based Gauss diagram obtained by smoothing $G_b$ along $\a$.
Then, $\ang{\C_{1},H_b}\equiv \ang{\C'_{1},H_b}\equiv 0\bmod p$ and $\ang{\C_{2n-1},H_b}=0$ for $n\geq 2$.
\end{lemma}

\begin{proof}
First, $\ang{\C_{1},H_b}=0$ is obvious by definition.
Next, by Lemma~\ref{lem:index}, we have $\ang{\C_{1},H_b}\equiv \ang{\C'_{1},H_b}\bmod p$.
Let $H'$ be a $1$-component subdiagram of $H_b$ with $2n-1$ arrows.
$H'$ must have an arrow connecting two circles of $H_b$.
Since the tail of the arrow is attached to the circle containing $b$, $H'$ is not ascending, and hence $\ang{\C_{2n-1},H_b}=0$.
\end{proof}

\begin{lemma}
\label{lem:det1}
Let $K$ be a checkerboard colorable virtual knot admitting a descending diagram.
Then, $\det K=1$.
\end{lemma}

\begin{proof}
First, a descending diagram is welded equivalent to the trivial one by \cite[Proposition~2.2]{Sat18}.
Since $\det K$ is invariant under welded equivalence, we conclude that $\det K=1$.
\end{proof}

\begin{lemma}
\label{lem:L1onL2}
Let $L=L_1\cup L_2$ be a $2$-component mod $2$ almost classical oriented virtual link.
If $L$ admits a diagram $D=D_1\cup D_2$ such that $D_1$ is above $D_2$ at every real crossing between them, then $\det L=0$.
\end{lemma}

\begin{proof}
Let $c_1,\dots,c_k$ be the real crossings of $D_1$ and itself, and let $c_{k+1},\dots,c_n$ be the rest of real crossings of $D$.
We may assume $k\geq 1$ and $n\geq 2$ by Reidemeister moves.
Let $a_1,\dots,a_l$ be long arcs of $D_1$ and let $a_{l+1},\dots,a_m$ be that of $D_2$.
The assumption implies $k=l$ and we have 
\[
B(D) =
\begin{pmatrix}
 B(D_1) & O \\
 \ast & \ast 
\end{pmatrix}.
\]
Since $\det B(D_1) =0$ (see a sentence before \cite[Proposition~3.1]{BoKa22}), any $(n-1)\times(n-1)$ minor of $B(D)$ is zero.
\end{proof}

\begin{theorem}
\label{thm:det}
Let $K$ be a checkerboard colorable oriented virtual knot, $G$ a based Gauss diagram of $K$, and $S$ a mock Seifert matrix for $K$.
Then, $\det S\equiv \pm\nabla_\asc(G)(2) \bmod 8$.
\end{theorem}

\begin{proof}
For non-negative integers $n$, we introduce the following proposition $(\ast)_n$ which is a refinement of the statement above.
\begin{description}
\item[$(\ast)_n$] 
Let $D$ be a diagram of a checkerboard colorable oriented virtual knot with $n$ real crossings, $G$ the Gauss diagram corresponding to $D$ with an arbitrary basepoint, and $S$ a mock Seifert matrix for $K$.
Then, $\det S\equiv \pm\nabla_\asc(G)(2) \bmod 8$.
Moreover, let $D_0$ be a diagram of $2$-component link obtained by smoothing a real crossing of a diagram of a checkerboard colorable oriented virtual knot with $n+1$ real crossings, $G_0$ a based Gauss diagram of $D_0$, and $S_0$ a mock Seifert matrix for the link represented by $D_0$.
Then, $\det S_0\equiv \pm\nabla_\asc(G_0)(2) \bmod 4$.
\end{description}
We prove $(\ast)_n$ by induction on $n$.
First, $(\ast)_0$ is true since a virtual knot without real crossings is the unknot and the link obtained from a knot with a single real crossing by smoothing the crossing is a trivial link.

We next assume that $(\ast)_{n-1}$ is true and we prove
\begin{align}
\label{eq:mod8}
\det S\equiv \pm\nabla_\asc(G)(2) \bmod 8
\end{align}
by induction on the warping degree $d$ of $G$.
If $d=0$, then $D$ is descending, and thus \eqref{eq:mod8} holds by Lemmas~\ref{lem:det1} and \ref{lem:warp_zero}.
Assume that \eqref{eq:mod8} holds when the warping degree is less than $d$.
Let $D$ and $S$ be a diagram and matrix in $(\ast)_n$ and assume $d(G)=d$.

Let $\a$ be an arrow of $G$.
By smoothing $G$ along $\a$, we obtain diagrams $G'$ and $G_0$.
Since $d(G')=d-1$, one has $\det S'\equiv \varepsilon'\nabla_\asc(G')(2) \bmod 8$ for some $\varepsilon'\in\{\pm 1\}$.
By $(\ast)_{n-1}$, we have $\det S_0\equiv \pm\nabla_\asc(G_0)(2) \bmod 4$.
Let $\varepsilon$ be the sign of the crossing.
Using Lemma~\ref{lem:skein_even}, we have
\begin{align*}
\det S &= \det S' +\varepsilon 2\det S_0 \\
&\equiv \varepsilon'\nabla_\asc(G')(2) +2\varepsilon \nabla_\asc(G_0)(2) \mod 8 \\
&\equiv \varepsilon'\nabla_\asc(G')(2) +2\varepsilon' \varepsilon \nabla_\asc(G_0)(2) \mod 8 \\
&\equiv \varepsilon'\nabla_\asc(G)(2) \mod 8.
\end{align*}
Here the third equality follows from
\[
\nabla_\asc(G_0)(2) = \sum_{i\geq 1}\ang{\C_{2i-1},G_0}2^{2i-1} \equiv \varepsilon'\nabla_\asc(G_0)(2) \mod 4.
\]

We finally prove
\begin{align}
\label{eq:mod4}
\det S_0\equiv \pm\nabla_\asc(G_0)(2) \bmod 4
\end{align}
by induction on the number $k$ of arrows in $G$ from the component without basepoint to the other component.
If $k=0$, then one component of $D_0$ is above the other component at every real crossing, and thus \eqref{eq:mod4} holds by Lemmas~\ref{lem:smoothing} and \ref{lem:L1onL2}.
Assume \eqref{eq:mod4} holds less than $k$.
Let $S_0$ and $G_0$ be a matrix and diagram in $(\ast)_{n}$, respectively.
Let $\a$ be an arrow in $G$ from the component without basepoint to the other component.
Let $D_{0}'$ and $D_{0}^0$ be diagrams obtained by crossing change and smoothing at the crossing corresponding to $\a$.
By the induction hypothesis, we have $\det S_{0}'\equiv \varepsilon'\nabla_\asc(G_{0}')(2) \bmod 4$ for some $\varepsilon'\in \{\pm 1\}$.
By $(\ast)_{n-1}$, we have $\det S_{0}^0\equiv \nabla_\asc(G_{0}^0)(2) \mod 8$.
Using Theorem~\ref{thm:det_skein}, we have
\begin{align*}
\det S^{0} &= \det S_{0}' +\varepsilon 2\det S_{0}^0 \\
&\equiv \varepsilon'\nabla_\asc(G_{0}')(2) +2\varepsilon \nabla_\asc(G_{0}^0)(2) \mod 4 \\
&\equiv \varepsilon'\nabla_\asc(G_{0}')(2) +2\varepsilon' \varepsilon \nabla_\asc(G_{0}^0)(2) \mod 4 \\
&\equiv \varepsilon'\nabla_\asc(G_{0})(2) \mod 4,
\end{align*}
where the last equality follows from Lemma~\ref{lem:skein_odd}.
\end{proof}

Now, we divide Theorem~\ref{thm:main} into the following two theorems and give their proofs.

\begin{theorem}
\label{thm:C_prime}
Let $G$ be a based Gauss diagram of a mod $p$ almost classical oriented knot.
Then, $\ang{\C_2,G}\equiv \ang{\C'_2,G}\bmod p$.
\end{theorem}

\begin{proof}
The proof is by induction on the pair $(n(G),d(G))$, where we write $n(G)$ for the number of arrows of $G$ in this proof.
In the case $(n(G),d(G))=(0,0)$, we see $\ang{\C_2,G}=\ang{\C'_2,G}=0$.
Assume that it holds for $G$ such that $(n(G),d(G))$ is smaller than a given $(n,d)$ in the lexicographic order.
Let $G$ be a Gauss diagram of a mod $p$ almost classical oriented knot with $(n(G),d(G)) = (n,d)$.
If $d(G)=0$, Lemma~\ref{lem:warp_zero} implies $\ang{\C_2,G}\equiv \ang{\C'_2,G}\bmod p$.

Consider the case $d(G)\geq 1$.
Then there exists an arrow $\a$ such that the head of $\a$ appears before the tail when we move on the circle from $b$.
Let $G_\a'$ (resp.\ $G_\a^0$) be the diagram obtained from $G$ by switching the orientation of $\a$ (resp.\ smoothing along $\a$).
By Lemma~\ref{lem:skein_even}, we have
\begin{align*}
\ang{\C_2, G} = \ang{\C_2, G_\a'}+\varepsilon \ang{\C_1, G_\a^0}, \\
\ang{\C'_2, G} = \ang{\C'_2, G_\a'}+\varepsilon \ang{\C'_1, G_\a^0}.
\end{align*}
for some $\varepsilon\in \{\pm 1\}$.
By the induction hypothesis and $I(\a)\equiv 0 \bmod p$, we conclude that $\ang{\C_2,G}\equiv \ang{\C'_2,G}\bmod p$.
\end{proof}

\begin{figure}[h]
 \centering \includegraphics[width=0.4\textwidth]{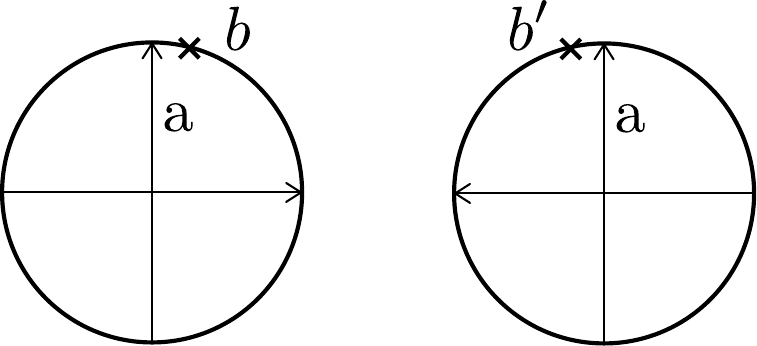}
 \caption{Ascending diagrams with basepoints $b$ and $b'$.}
 \label{fig:b_bprime}
\end{figure}

\begin{theorem}
\label{thm:basepoint}
Let $G$ be a Gauss diagram of a mod $p$ almost classical oriented knot.
Then $\ang{\C_2,G_b}\bmod p$ and $\ang{\C'_2,G_b}\bmod p$ do not depend on the choice of a basepoint $b$.
\end{theorem}

\begin{proof}
By Theorem~\ref{thm:C_prime}, we need to discuss only $\ang{\C_2,G_b}\bmod p$.
It suffices to show that $\ang{\C_2,G_b}\equiv \ang{\C_2,G_{b'}}\bmod 2$, where the basepoint $b$, an endpoint of an arrow $\a$, and $b'$ are aligned in this order on the oriented circle.
We use the same symbols $r_\pm$ and $l_\pm$ as Definition~\ref{def:index} with respect to $\a$ with sign $\varepsilon$.
There are two cases where the above endpoint of $\a$ between $b$ and $b'$ is (I) the head or (II) the tail of $\a$.

In the case (I), the contribution of subdiagrams containing $\a$ in $\ang{\C_2,G_b}$ (resp.\ $\ang{\C_2,G_{b'}}$) is $\varepsilon(r_+ - r_-)$ (resp.\ $\varepsilon(l_+ - l_-)$) as drawn in Figure~\ref{fig:b_bprime}.
They are the same modulo $p$ by Lemma~\ref{lem:index}.
Since the contribution of subdiagrams which do not contain $\a$ in $\ang{\C_2,G_b}$ is the same as that of $\ang{\C_2,G_{b'}}$, we conclude that $\ang{\C_2,G_b}\equiv \ang{\C_2,G_{b'}}\bmod 2$.

In the case (II), there is no subdiagram containing $\a$ which contributes to $\ang{\C_2,G_b}$.
Therefore, we complete the proof.
\end{proof}

\begin{proof}[Proof of Corollary~\ref{cor:det}]
Let $G$ be a based Gauss diagram of a checkerboard colorable virtual knot $K$.
It follows from Theorems~\ref{thm:det} and \ref{thm:main} that
\begin{align*}
\det K &= \pm\nabla_\asc(G)(2) 
= \pm\sum_{i\geq 0}\ang{\C_{2i},G}2^{2i} \\
&\equiv \pm 1+\ang{\C_{2},G}4 
\equiv \pm 1+4v_2(K) \mod 8.
\end{align*}
This completes the proof.
\end{proof}

Theorem~\ref{thm:main} and Corollary~\ref{cor:det} reveal important properties of the second-degree term of the ascending/descending polynomial for checkerboard colorable virtual knots.
On the other hand, higher degree terms of the ascending/descending polynomial are not necessarily well known.
It is natural to investigate their properties and applications.

\end{document}